\documentclass[11pt]{article}
\usepackage{amsmath, amsthm, amssymb}
\usepackage{amsmath}
\usepackage{amsthm}
\usepackage{amscd,amssymb}
\usepackage[all]{xy}
\usepackage{color}
\usepackage{comment}
\usepackage{appendix}
\usepackage{verbatim}
\usepackage{mathrsfs}
\usepackage{pdfsync}
\date{\today}
\usepackage{graphicx,epsfig}
\newtheorem{thm}{Theorem}[section]
\newtheorem{cor}[thm]{Corollary}
\newtheorem{lem}[thm]{Lemma}

\let\oldproofname=\proofname
\renewcommand{\proofname}{\rm\bf{\oldproofname}}

 \newtheorem{prop}[thm]{Proposition}
  
\theoremstyle{definition}
 
 \newtheorem{rmk}[thm]{Remark}
 \newtheorem{defn}{Definition}[section]
 \newtheorem{claim}[thm]{Claim}
  
 \newtheorem{conj}{Conjecture}[section]

\newcommand{\Z}{\mathbb{Z}}
\newcommand{\A}{\mathcal{A}}

\newcommand{\N}{\mathcal{N}}
\newcommand{\M}{\mathcal{M}}

\def\Hom{\text{Hom}}

\begin{document}
\title{Neighborhood Complexes of Some Exponential Graphs}
\author{ Nandini Nilakantan\footnote{Department of Mathematics and Statistics, IIT Kanpur, Kanpur-208016, India. nandini@iitk.ac.in.},  Samir Shukla\footnote{{Department of Mathematics and Statistics, IIT Kanpur, Kanpur-208016, India. samirs@iitk.ac.in.}}}
\maketitle
\begin{abstract}
In this article, we consider the bipartite graphs $K_2 \times K_n$. We first show that the connectedness of  $\displaystyle \N(K_{n+1}^{K_n}) =0$.
 Further, we show that  $\text{Hom}(K_2 \times K_{n}, K_{m})$ is homotopic to $S^{m-2}$, if $2\leq m <n$.
\end{abstract}


\vspace{.1in}

\hrule

\section{Introduction}

Determining the chromatic number of a graph is a classical problem
in graph theory and finds applications in several fields. The
Kneser conjecture posed in 1955 and solved by Lov{\'a}sz \cite{l} in
1978, dealt with the problem of computing the chromatic number of a
certain class of graphs, now called the {\it Kneser graphs}. To prove this
conjecture, Lov{\'a}sz first constructed the
neighborhood complex $\mathcal{N}(G)$  of a graph $G$, which is a
simplicial complex and then related the connectivity of this complex
to the chromatic number of $G$.

 Lov{\'a}sz  introduced the notion of a simplicial complex called the {\it  Hom complex}, denoted by $\Hom(G,H)$ for graphs $G$ and $H$,
  which generalized the notion of a neighborhood complex. In particular, $\Hom(K_2,G)$ (where $K_2$ denotes a complete graph with $2$ vertices)
   and $\N(G)$ are homotopy equivalent. The idea was to be able to estimate the chromatic number of an arbitrary graph $G$ by understanding the connectivity of
   the Hom complex from some standard graph into $G$. Taking $H$ to be the complete graph $K_{n}$  makes each of the complexes $\text{Hom} (G, K_n)$ highly connected.
     In \cite{BK} Babson and Kozlov made the following conjecture.

\begin{conj}
For a graph $G$ with maximal degree $d$, $\text{Hom}(G,K_n)$ is  atleast $n-d-2$ connected.

\end{conj}

In \cite{CK}, {\v C}uki{\' c} and Kozlov presented a proof for the above conjecture. 
They further showed that in the case when $G$ is an odd cycle,  $\text{Hom}(G,K_n)$ is $n-4$ connected for all $n \geq 3.$
From \cite{dk1}, it is seen that for any even cycle $C_{2m}$,  $\text{Hom}(C_{2m},K_n)$ is $n-4$ connected for all $n \geq 3.$

It is natural to ask whether it is possible to classify the class of graphs $G$ for which the Hom complexes $\text{Hom}(G,K_n)$
 are exactly $n-d-2$ connected. In this article, we consider the bipartite graphs $K_2\times K_n$, which are $n-1$ regular graphs.
  Since $\displaystyle \text{Hom}(K_{2}\times K_{n}, K_m) \simeq \text{Hom}(K_{2}, K_m^{K_{n}})$,
   it is sufficient to determine the connectedness of $\text{Hom}(K_{2}, K_m^{K_{n}})$ which is the same as the connectedness of $\N (K_{m}^{K_n})$.
   The main results of this article are

 \begin{thm} \label{thm1}

 Let $n \geq 4$ and $p$ = $\frac{n! (n-1) n} {2}$. Then  

   \begin{center}
            $H_k(\mathcal{N}(K_{n+1}^{K_n}); \Z_2) = \begin{cases}
            \ \mathbb{Z}_2  & \text{if \hspace{0.3 cm}$k = 0, 1$ or $n-1$},\\
            \ \mathbb{Z}_2^{p - n! +1} & \text{ if \hspace{0.3 cm}$k = 2$},\\
               \              0 & \text{otherwise}.\\
                       \end{cases}$
  \end{center}
  \end{thm}

\begin{thm}\label{thm2}

 conn$(\mathcal{N}(K_{n+1}^{K_n}))$ = 0 for all $n \geq 2$.
\end{thm}

\begin{cor}\label{cor1}

 Let $n \geq 2$ and $2 \leq m \leq n+1$. Then
 \begin{center}
   $\text{conn}(\text{Hom}(K_2 \times K_n, K_m)) = \begin{cases}
            \ 0 & \text{if \hspace{0.3 cm}$m=n+1$},\\
            \ -1 & \text{if \hspace{0.3 cm} $m = n$},\\
                             m-3 & \text{otherwise}.\\
                       \end{cases}$
 \end{center}

\end{cor}

We make the following conjecture.

\begin{conj}
The lower bounds given in \cite{CK} are exact for all bipartite graphs of the type $K_2\times K_n$.
\end{conj}

\section{Preliminaries}

A  graph $G$ is a  pair $(V(G), E(G))$,  where $V(G)$  is the set of vertices of $G$  and $E(G) \subset V(G)\times V(G)$ denotes the set of edges.
 If $(x, y) \in E(G)$, it is also denoted by $x \sim y$ and $x$ is said to be adjacent to $y$.
The {\it degree} of a vertex $v$ is defined  as $\text{deg}(v) =
|\{y \in V(G) \ | \ x \sim y\}|$ ($| X|$ represents the
cardinality of the set $X$).

\begin{itemize}
\item A {\it  bipartite graph} is a graph $G$ with subsets  $X$ and $Y$ of $V(G)$ such that $V(G)=X   \sqcup Y$ and $(v,w) \notin E(G)$ if $\{v,w\} \subset X$ or $\{v,w\} \subset Y$.

A standard example of bipartite graphs are the even cycles $C_{2n}$
where $V(C_{2n})=\{1,2,\dots, 2n\}$ and $E(C_{2n})=\{(i,i+1) \ | \
1\leq i \leq 2n-1\}\cup\{(1,2n)\}$. In this case $X=\{1,3,5,\ldots
,2n-1\}$ and $Y=\{2,4,6,\dots,$ $  2n\}$.

\item A {\it graph homomorphism} from  $G$ to $H$ is a function $\phi: V(G) \to V(H)$ such that, $$(v,w) \in E(G) \implies (\phi(v),\phi(w)) \in E(H).$$

\item
A {\it finite abstract simplicial complex X} is a collection of
finite sets where $\tau \in X$ and $\sigma \subset \tau$,
implies $\sigma \in X$. The elements  of $X$ are called the  {\it simplices}
of $X$. If $\sigma \in X$ and $|\sigma |=k+1$, then $\sigma$ is said
to be $k\,dimensional$. A $k-1$ dimensional subset of a $k$ simplex
$\sigma$ is called a {\it facet} of $\sigma$.

\item
Let $v$ be a vertex of a graph $G$. The {\it
neighbourhood  of $v$ }is defined as $N(v)=\{ w \in V(G) \ |  \
(v,w) \in E(G)\}$.  If $A\subset V(G)$, the neighbourhood $A$
is defined as $N(A)= \{x \in  V(G) \ | \ (x,a) \in E(G)
\,\,\forall\,\, a \in A \}$.

\item
The {\it neighborhood complex} $\N(G)$ of a graph $G$ is the abstract simplicial complex whose simplices are those subsets of 
vertices of $G$, which have a common neighbor.

\item Let $G$ be a graph and $N(u) \subset N(v)$  for two distinct vertices $u$ and $V$ of $G$. The graph $G\setminus \{u\}$ is called a {\it fold} of $G$.  Here,  $V(G\setminus
\{u\}) = V(G) \setminus \{u\}$ and the edges in the subgraph
$G\setminus \{u\}$ are all those edges of $G$ which do not contain
$u$.

\item
  Let  $X$ be a simplicial complex and $\tau, \sigma \in X$ such that
$\sigma \subsetneq \tau$ and  $\tau$ is the only maximal simplex in $X$ that contains $\sigma$.
  A  {\it simplicial collapse} of $X$ is the simplicial complex $Y$ obtained from $X$ by
  removing all those simplices $\gamma$  of $X$ such that
  $\sigma \subseteq \gamma \subseteq \tau$. $\sigma$ is called a {\it free face} of $\tau$ and $(\sigma, \tau)$ is called a {\it collapsible pair} and is denoted by $X \searrow Y$.

\item
For any two graphs $G$ and $H$, $\Hom(G,H)$ is the polyhedral
complex whose cells are indexed by all functions $\eta: V(G) \to
2^{V(H)}\setminus \{\emptyset\}$, such that if $(v,w) \in E(G)$ then
$\eta(v) \times \eta(w) \subset E(H)$.

 Elements of  $\Hom(G,H)$ are called cells and are denoted by $(\eta(v_1), \dots \\ \eta(v_k))$, where $V(G)=\{v_1, \dots, v_k\}$.
A cell $(A_1, \dots, A_{k})$ is called a {\it face} of $B=(B_1,
\dots B_k)$, if $A_{i} \subset B_{i}$ $\forall\, 1\leq i\leq k$. The
Hom complex is often referred to as a topological space. Here, we
are referring to the geometric realisation of the order complex of
the poset. The simplicial complex whose simplices are the chains of
the Poset $P$ is called the order complex of $P$.

\item
A  topological space $X$ is said to be $n$ {\it connected} if $\pi_\ast(X)= 0$ for all $ \ast \leq n$.\\
 By convention, $\pi_0(X)= 0 $ means $X$ is connected.
The connectivity of a topological space $X$ is denoted by
$\text{conn}(X)$, {\it i.e.}, $\text{conn}(X)$ is the largest
integer $m$ such that $X$ is $m$ connected. If $X$ is a non empty
disconnected space, it is said to be $-1$ connected and if it is
empty, it is said to be $-\infty$  connected.
\end{itemize}

We now review some of the constructions related to the existence of an {\it internal hom} which is related to the categorical product.  Details can be found in \cite{pj, cj, sm}.

\begin{itemize}
\item
The {\it categorical product} of two graphs $G$ and $H$, denoted by $G\times H$ is the graph where $V(G\times H)=V(G)\times V(H)$ and  $(g,h) \sim (g',h')$ in $G\times H$  if $g \sim g'$  and $h \sim h'$ in $G$ and $H$ respectively.

\item
If $G$ and $H$ are two graphs, then the {\it exponential graph} $\displaystyle H^{G}$ is defined to be the graph where $\displaystyle V(H^{G})$ contains all the set maps from $V(G)$ to $V(H)$.  Any two vertices  $f$  and $f'$ in $\displaystyle V(H^{G})$ are  said to be adjacent, if $ v \sim v'$  in $G$ implies that $f(v) \sim f'(v')$ in $H$.

\end{itemize}

Using tools from poset topology (\cite{bj}), it can be shown that given a poset $P$ and a poset map $c:P \rightarrow P$ such that $ c\circ c =c$ and $c (x) \geq x, \forall \,x \in P$, there is a strong deformation retract induced by $c :P \rightarrow c(P)$ on the relevant spaces. Here, $c$ is called the {\it closure map}.

From \cite[Proposition 3.5]{ad} we have a relationship between the exponential graph and the categorical product in the $\text{Hom}$-complex.

\begin{prop}
Let $G$, $H$ and $K$ be graphs. Then  $\displaystyle \text{Hom}(G \times H, K)$ can be included in $\displaystyle \text{Hom}(G, K^{H})$  so that $\displaystyle \text{Hom}(G \times H, K)$ is the image of the closure map on $\displaystyle \text{Hom}(G, K^{H})$. In particular, there is a strong deformation retract  $\displaystyle |\text{Hom}(G \times H, K)| \hookrightarrow |\text{Hom}(G, K^{H})|$.
\end{prop}

From \cite[Proposition 5.1]{BK} we have the following result which allows us to replace a graph by a subgraph in the $\text{Hom}$ complex.

\begin{prop}\label{fold}
Let $G$ and $H$ be graphs such that $u, v $ are distinct vertices of $G$ and $N(u) \subset N(v)$. The inclusion $i :G \setminus \{u\}  \hookrightarrow G$ respectively,
 the homomorphism  $\phi:G \rightarrow G \setminus \{u\} $ which maps $v$ to $u$ and fixes all the other vertices,
  induces the homotopy equivalence  $i_{H} :\text{Hom} (G,H)\rightarrow \text{Hom}(G\setminus \{u\},  H)$,‪  respectively  $\phi_{H} :\text{Hom}(G\setminus \{u\},  H) \rightarrow \text{Hom}(G,H)$.
\end{prop}

\section{Tools from Discrete Morse Theory}

We introduce some tools from Discrete Morse Theory  which have been
used in this article. R. Forman in \cite{f} introduced what has now
become a standard tool in Topological Combinatorics, Discrete Morse
Theory. The principal idea of Discrete Morse Theory (simplicial) is
to pair simplices in a complex in such a way that they can be
cancelled by elementary collapses. This will reduce the original
complex to  a homotopy equivalent complex, which is not necessarily
simplicial, but which has fewer cells. More details of discrete
Morse theory can be found in \cite{jj} and \cite{dk}.

\begin{defn}
A {\it partial matching} in a poset $P$ is a subset $\M$ of $P \times P$ such that
\begin{itemize}
\item $(a,b) \in \M$ implies $b\succ a$, {\it i.e. $a<b$ and $\not\exists \,c$ such that $a<c<b$}.
\item Each element  in $P$ belongs to at most one element in $\M$.
\end{itemize}
\end{defn}

In other words, if $\M$ is a {\it partial matching} on a poset $P$
then there exists  $A \subset P$ and an injective map $f: A 
\rightarrow P\setminus A$ such that $x \prec f(x)$ for all $x \in A$. 

\begin{defn}
An {\it acyclic matching} is a partial matching  $\M$ on the Poset $P$ such that there does not exist a cycle
\begin{eqnarray*}
x_1 \prec f(x_1)  \succ x_2 \prec  f( x_2) \succ x_3 \dots \prec f(x_t) \succ x_1, t\geq 2.
\end{eqnarray*}

\end{defn}

Given an acyclic partial matching on $P$, those elements of $P$ which do not belong to the matching are said to be {\it critical }. To obtain the  desired homotopy equivalence, the following result is used.

\begin{thm} (Main theorem of Discrete Morse Theory)\label{dmt}\cite{f}
\mbox{}

Let $X$ be a simplicial complex and let $\A$ be an acyclic matching such that the empty set is not critical. Then, $X$ is homotopy equivalent to a cell complex which has a $d$ -dimensional cell for each $d$ -dimensional critical face of $X$ together with an additional $0$-cell.

\end{thm}

\section{Main Result}

To prove the Theorems \ref{thm1} and  \ref{thm2}, we first construct
an acyclic matching on the face poset of $\N (K_{n+1}^{K_{n}})$
after which we construct the Morse Complex corresponding to this
acyclic matching and use this complex to compute the homology
groups.

In this article $n \geq 3$ and $[n]$ denotes the
set $\{1,2,\dots, n\}$. Any vertex in the exponential graph
$K_{n+1}^{K_n}$ is a set map $f: K_{n} \rightarrow K_{n+1}$.

\begin{lem} \label{lemma2}

The graph $K_{m}^{K_n}$  can be folded onto the graph $G$, where the vertices
$f \in V(G)$ have images of cardinality either 1 or $n$.

\end{lem}
\begin{proof} Consider the vertex $f$ such that  $1<|$Im $f|$ $< n$.  Since $f$ is not injective
there exist distinct $i, j \in [n]$ such that $f(i)=f(j) = \alpha$.
Consider  $\tilde{f} \in V(K_{m}^{K_n})$ such that
$\tilde{f}([n])=\alpha$. By the definition of the exponential graph, any
neighbour $h$  of $f$ will not have $\alpha$ in its image  and
therefore $h$ will be a neighbour of $\tilde{f}$ thereby showing
that $N(f) \subset N(\tilde{f})$.
 $K_{m}^{K_n}$ can be folded to the subgraph $K_{m}^{K_n}\setminus \{f\}$.
  Repeating the argument for all noninjective, non constant maps from $[n]$ to $[m]$,   $K_{m}^{K_n}$
can be folded to the graph $G$ whose vertices are either constant or injective maps from $[n]$ to $[m]$.
\end{proof}

From Proposition \ref{fold}, we observe that $\N (K_{n+1}^{K_n})  \simeq \N (G)$. Hence, it
is sufficient to study the homotopy type of $\N (G)$.

Henceforth, if $f \in V(G)$ and $f([n])=\{x\}$,  $f$ shall be denoted by $<x>$. In the other cases  the string $a_1a_2\dots a_n$ will denote the vertex   $f$ where $a_i =f(i)$,  $1\leq i \leq n.$ Hence, if the notation $a_1a_2\dots a_{n}$ is used, it is understood that  for $1\leq i<j\leq n$, $a_i \neq a_j$

Let $f =a_1a_2 \ldots a_n \in V(G)$ and $x \notin
\text{Im}\,f$. Define $A_i^f$ to be the set $\{a_1,a_2, \ldots,\hat{a_i},\ldots,
a_n\} = \{a_1, \ldots, a_n\} \setminus \{a_i\}$.
 The map $f_{k}$ is defined on $[n]$ by
  \begin{center}
            $f_k(i)= \begin{cases}
            f(i) , & \text{if $k \neq i $},\\
x, & \text{if $k = i$}.\\
                       \end{cases}$
  \end{center}

 We first consider the maximal simplices of $\N(G)$.

\begin{lem} \label{lemmab}
Let $f \in V(G)$. Then
\begin{enumerate}
 \item [(i)]$f=a_1a_2\dots a_n, \,x \notin \text{Im}\,f \Rightarrow N(f) = \{f, <x>, f_1,
f_2, \ldots, f_n\}$.
\item [(ii)]
$f=<x> \ \Rightarrow N(f) = \{ <y>|\, y \neq x\}\cup\{g \in V(G)| x
\notin \text{Im}\,g\}$.
\end{enumerate}
\end{lem}

\begin{proof} \mbox{}
\begin{enumerate}
\item [(i)] Since $a_i \neq a_j$, $\forall \ i \neq j$, $f \sim f$.  If $g= \ <x>$, then $g(i) \neq f(j)$ for $i\neq j$ and thus $<x> \ \in N(f)$.
For any $l \in [n]$, $f_{l}(i) \neq f(j)$ for $i\neq j$ which implies $f_{l}\in N(f)$. Thus $\{f, <x>, f_1,
f_2, \ldots, f_n\} \subset N(f)$.
Conversely if  $\tilde{f} \in N(f)$, then $\tilde{f} (i) \in \{a_i,
x\}$. Since $|\text{Im}\,\tilde{f}|=1$ or $n$, if $\tilde{f} \neq \
<x>$ then $\tilde{f}$ has to be  $f$ or $f_{l}$ for some $l \in[n]$.

\item [(ii)]Let $f = \ <x>$. Clearly, $<y> \ \sim \ <x>$ for all $y \neq x$. If $g \in V(G)$ and $x \notin \text{Im}\, g$, then $g \ \sim\  <x>$.
Conversely if $\tilde{f} \in N(f)$,  then  $x$ cannot belong to the
image of  $\tilde{f} $. Since $| \text{Im}\,\tilde{f}|$ has to be either
$1$ or $n$ from Lemma \ref{lemma2}, the proof follows.
\end{enumerate}
\end{proof}

We  now determine the free faces in $\N(G)$.

\begin{lem} \label{lemma3}
 Let $f \in V(G)$. Then
 \begin{enumerate}
  \item [(i)] $f=a_1a_2\dots a_n \Rightarrow (\{f_s,f_t\}, N(f))$ is a collapsible pair
  $\forall$ $1 \leq s < t \leq n$.

  \item [(ii)] $f = \, <y>$,  $g\neq \tilde{g} \in V(G)$ non constant neighbours of $f
  \Rightarrow$ $(\{g, \tilde{g}\},$ $  N(f))$ is  a collapsible pair.
 \end{enumerate}
\end{lem}
\begin{proof} \hspace{10cm}
\begin{enumerate}
\item [(i)] From Lemma \ref{lemmab},
  $N(f)$ = $\{f, <x>, f_1, f_2, \ldots, f_n\}$, where $x \notin \text{Im} f$. All the maximal simplices of $\mathcal{N}(G)$
  are of the form $N(g)$,
  where $g \in V(G)$. Suppose there exists $\tilde{f} \in V(G)$
  such that $\{f_s, f_t\} \subset N(\tilde{f})$,  for some $1 \leq s < t \leq n$.
  Since $A_s^{f_s} = [n+1] \setminus
  \{a_s, x\}$ and $A_i^{f_s}$ = $[n+1] \setminus \{a_i, a_s\}$ if $i \neq s$, then for each $i \in [n]$ at least one of the sets
  $A_i^{f_s}$ or $A_i^{f_t}$ contains $x$. Therefore
  $A_i^{f_s} \cup A_i^{f_t}=[n+1] \setminus \{a_i\}\,\forall\,1 \leq i \leq n$. Since $\tilde{f}$ is a neighbour of both
  $f_s$ and  $f_t$,   $\tilde{f}(i) \neq f_{s}(j), f_{t}(j)\,\forall \,j \neq i$, which implies that
  $\tilde{f}(i) = a_i= f(i)$. Hence $\{f_s, f_t\}$ is free in $N(f)$.

 \item[(ii)] Since $g, \tilde{g}$ are neighbours of $f$ and $i \sim j$ in $K_n$ $\forall$ $i \neq j$, $ f(j) \neq g(i),\,\tilde{g}(i)$
implies that $y \notin \text{Im}\, g,\, \text{Im} \,\tilde{g}$,
which shows that $\text{Im}\,g = \text{Im} \,\tilde{g}$.
  Let $h \in V(G)$ such that $g, \tilde{g} \in N(h)$. Since $g, \tilde{g}$ are distinct and injective there exist $s \neq t \in [n]$ such that
 $g(s) \neq \tilde{g}(s)$ and $g(t) \neq \tilde{g}(t)$. $A_s^{g}$ = $[n+1] \setminus \{y, g(s)\}$ and
 $A_s^{\tilde{g}}$ = $[n+1]
 \setminus \{y, \tilde{g}(s)\}$. Since $g(s) \neq \tilde{g}(s)$, $A_s^{g}$ $\cup$ $A_s^{\tilde{g}}$ = $[n+1]
 \setminus \{y\}$.
 $h$ is a neighbour of $g$ and $\tilde{g}$ and $i \sim s$ in $K_n$ $\forall$ $i \neq s$,
 implies $h(s) \neq g(i), \tilde{g}(i)$.
 In particular, $h(s) \notin A_s^g \cup A_s^{\tilde{g}}$ and therefore $h(s) = y$ (similarly $h(t) = y$).
Therefore $h(i)=y$ $\forall \, i \in [n]$ and is equal to $f$.
  \end{enumerate}
\end{proof}

Let $M(X)$ be the set of maximal simplices in the simplicial complex $X$.

\begin{lem} \label{lemma1}
In a  simplicial complex $X$, let $\sigma =\{x_1,x_2,\dots, x_t, y_1, y_2, \dots, $ $  y_k\}$, $t\geq 2$
 be a maximal simplex such that $\{x_i,x_j\}$ is a free face of $\sigma$  for $1 \leq i < j \leq t$.  $X$
 collapses to the subcomplex $Y$ where $M(Y)$ = $M' \cup\,  \{ \{x_i, y_1, \dots, y_k\} |1 \leq i\leq t\}$ and $M'=M(X) \setminus \{\sigma\}$.
\end{lem}

\begin{proof}   We first consider collapses with the faces $\{x_{1}, x_j\}$, $2 \leq j \leq t$.

\begin{claim} \label{clm1}
$X \searrow X'$ with $M(X')=M' \cup  \{x_{1}, y_{1}, \dots, y_{k}\} \cup \{\sigma \setminus\{x_1\} \}$.

Since $\{x_1,x_2\}$ is a free face of $\sigma$, $X  \searrow X_{12}$ with $M(X_{12})= M'\cup \{\sigma \setminus \{x_1\}\} \cup \{ \sigma \setminus \{x_2\}\}$. In $X_{12}$, $\{x_1, x_3\}$ is a free face  of $\sigma \setminus \{x_2\}$ and hence $X  \searrow X_{13}$ with $M(X_{13})= M'\cup \{\sigma \setminus \{x_1\} \}\cup \sigma \setminus \{x_2, x_3\}.$
Inductively, we assume that  $M(X_{1l})= M'\cup\{\sigma \setminus \{x_1\}\} \cup \sigma \setminus \{x_2, \dots x_l\}.$ In  $X_{1l}$, $\{x_1, x_{l+1}\}$ is a free face of  $\sigma \setminus \{x_{2}, \dots x_{l}\}$. Hence $X_{1l}  \searrow X_{1l+1}$ with $M(X_{1l+1})=M' \cup \{\sigma \setminus \{x_1\}\}\cup \sigma \setminus \{x_2, \dots x_{l+1}\}.$ This proves the claim.
\end{claim}

For $2\leq i \leq t-1$, considering the pairs $\{x_{i}, x_{j}\}$, $i+1\leq j\leq t$ and using Claim \ref{clm1}, the lemma follows.
\end{proof}

 The Lemmas \ref{lemma3} and \ref{lemma1} show that
 $\mathcal{N}(G)$ collapses to a subcomplex $\Delta_1$ with $M(\Delta_1)$ = $M_1 \cup M_2$, where
 \begin{eqnarray*}
 && M_1 = \{\{f, f_s, <x>\}\ | \ f \in V(G),\ x \notin \mbox{Im}\ f  \ \mbox{and} \  s \in [n]\} \,and \\
 && M_2 = \{\{<y_1>, <y_2>, \ldots, <y_n>, g\} \ | \
 \mbox{Im}\ g = \{y_1, y_2, \ldots, y_n\} \}.
  \end{eqnarray*}

 For any simplex $\sigma_g = \{<y_1>, <y_2>, \ldots, <y_n>, g\}$ in  $M_2$, if $1 \in \text{Im}\, g$, let $y_1=1$
  and if $ 1 \notin \text{Im}\, g$, let $y_1=2.$ In $\sigma_g $ for all $1 \leq i <j \leq  n$,
$\{<y_i>, <y_j>, g\}$ are free faces. Considering the faces
$\{<y_2>, <y_j>, g\}$, $3 \leq j \leq  n$, we get the following result.

\begin{claim} \label{clm2}
$\Delta_1 \searrow \Delta'$ and $M(\Delta')=M_1\cup \{M_{2}\setminus \{\sigma_{g}\}\} \cup \{\sigma_{g} \setminus \{g\}\} \cup \{\sigma_{g} \setminus\{ <y_{2}>\} \cup  \{<y_1>, <y_2>, g\}.$

\begin{proof} Let $Y= M_1\cup \{M_2\setminus \{\sigma_g\}\} \cup \{\sigma_g \setminus \{g\}\}$.
 Since $\{<y_2>, <y_3>, g\}$ is a free face of $\sigma_g $,  $\Delta_1 \searrow \Delta_{1,3}$
 where $M(\Delta_{1,3})= Y\cup \, \{\sigma_g \setminus \{<y_2>\}\}\, \cup \, \{\sigma_g \setminus \{<y_3>\}\}$.
In $\Delta_{1,3}$,  $\{<y_2>, <y_4>, g\}$ is a free face of $\{\sigma_g \setminus \{<y_3>\}\}$ and so  $\Delta_{1,3} \searrow \Delta_{1,4}$ where $M(\Delta_{1,4})= Y\cup \, \{\sigma_g \setminus \{<y_2>\}\}\, \cup \, \{\sigma_g \setminus \{<y_3>\,<y_4>\}\}$. Inductively, assume that   $\Delta_{1} \searrow \Delta_{1,n-1}$ where $M(\Delta_{1,n-1})= Y\cup \, \{\sigma_g \setminus \{<y_2>\}\}\, \cup \, \{\sigma_g \setminus \{<y_3>,\,<y_4>,\dots, <y_{n-1}>\}\}$.
In $\Delta_{1,n-1}$, $\alpha =\{<y_2>, <y_n>, g\}$ is a free face of 
$\sigma_g \setminus \{<y_3>,<y_4>, \dots, <y_{n-1}>\}$. By a simplicial collapse, we get the  complex $\Delta'$
where $M(\Delta')=Y \cup \, \{\sigma_{g} \setminus\{ <y_{2}>\} \}\,
\cup  \{<y_{1}>, <y_2>,g\}.\,$
\end{proof}
\end{claim}
For $3\leq i \leq n-1$, using Claim \ref{clm2} for the simplices $\{<y_{i}>,
<y_{j}>,g\}$, $i+1\leq j\leq n$, we get
$\Delta_1\searrow Z$ where $M(Z)= Y \cup \{\{<y_1>, <y_i>,
g\} |\,i \in \{2,3,4, \ldots, n\}\}$. Repeating the above argument
for the remaining elements of $M_2$, $\mathcal{N}(G)$ collapses to
the subcomplex $\Delta$, $M(\Delta)$ = $M_1 \cup A_1 \cup A_2
\cup A_3$, where
\begin{itemize}
 \item $A_1 = \{ \{<1>, <y>, g\}  \ | \ 1, y \in \mbox{Im}\ g \},$
 \item $A_2 = \{ \{<2>, <y>, g\}  \ | \ 2, y \in \mbox{Im}\ g\ \mbox{and} \ 1 \notin \mbox{Im}\ g  \}$ and
 \item $A_3 = \{ \{<y_1>, <y_2>, \ldots, <y_n>\} \ | \ y_1, y_2, \ldots, y_n \in [n+1]\}.$
 \end{itemize}
Since $\N (K_{n+1}^{K_n})$ is homotopy equivalent to $\N (G)$ which collapses to $\Delta$,  we have
$\N (K_{n+1}^{K_n}) \simeq \Delta.$ We now construct an acyclic matching on the face poset of $\Delta$  to compute the Morse Complex corresponding to this matching.

\medskip

Let $(P, \subset)$ denote the face poset of $\Delta$.
 Define $S_1\subset P$ to be the set $\{\sigma \in P \ |\ <1> \ \notin \sigma,\,
 \sigma\, \cup <1>\,\in \Delta \}$ and the map
 $\mu_1 : S_1 \rightarrow P \setminus S_1$ by $\mu_1(\sigma) = \sigma \cup <1>$.  $\mu_1$
  is injective and $(S_1, \mu_1(S_1))$ is a partial matching on $P$.
Let  $S'$ = $P \setminus (S_1 \cup \mu_1(S_1))$.

 \begin{lem}\label{4.7}
  Any 1-cell $\sigma$ of $S'$ will be of one of the following types.
  \begin{enumerate}
 \item [(I)] $\{f, <x>\}$, $1 \notin\text{Im}\, f$, $x \neq 1$.
 \item [(II)] $\{f,<x>\}$,  $ x \notin \text{Im}\, f$, $ x \neq 1$.
 \item [(III)] $\{f,f_i\}$,  $a_k= 1$,  $i \neq k$.
  \end{enumerate}
 \end{lem}

\begin{proof}  Let $\tau$ be a maximal simplex and $\sigma
\subsetneq \tau$. From the above discussion, $\tau$ has to be an element in one of the sets $M_1, A_1,A_2$ or $A_3$.

$\tau \notin A_1$ since $\sigma \notin \mu_1(S_1) \cup S_1$. If $\tau \in A_3$,
then $\sigma$ has to be of the form $\{<x>, <y>\}$ for some $x, y \neq 1$.
There exists $z \neq 1,x,y$ such that $\sigma$ $\cup <1>$ $\in N(<z>)$, which implies $\sigma
\in S_1$, a contradiction. Hence $\tau \notin A_3$.

If $\tau \in M_1$, then $\tau$ = $\{f, f_i, <x>\}$, where $f=a_1 a_2
\ldots a_n$ and $x\notin \text{Im}\,f$. Clearly $x \neq 1$,  as  if
$x=1$, then $ \sigma \in S_1 \cup \mu (S_1)$, a contradiction.  Let
$a_k =1$. Since $\sigma \subset \tau$, $\sigma=\{f, <x>\} , \ \{
f_i, <x>\}\ \text{or}\  \{f, f_i\}.$

\begin{enumerate}

\item [(i)] $\sigma$ = $\{f_i, <x>\}$.

If $i \neq k$, then $a_i \neq 1, x  \Rightarrow \sigma \ \cup <1>  \
\in N(<a_i>)  \Rightarrow \sigma \in S_1$, a contradiction.

If  $i = k$, then $A_k^{f_k} = [n+1] \setminus \{1,x\} \Rightarrow
N(f_k, <x>, <1>) = \emptyset$ and thus $\sigma$ is of the type
$(I)$.
\item [(ii)] $\sigma$ = $\{f, <x>\}$.\\
  Since $x \neq 1$, and  $A_k^{f} \cup\{x\} \cup \{1\} =[n+1]$,  $\sigma \notin S_{1} \cup \mu (S_1)$.
 Hence $\sigma\in S'$  and is of the type $(II)$.
\item [(iii)] $\sigma$ = $\{f, f_i\}$.\\
$a_k = 1 \Rightarrow 1 \notin \text{Im} f_k \Rightarrow \ <1> \ \in
N(f_k)\Rightarrow \{f, f_{k}, <1>\} \in N(f_k)$. Since $ \sigma \in
S'$, $\sigma$ cannot be $\{f, f_{k}\}$  and hence $i \neq k$.
$A_k^{f_i} = [n+1] \setminus \{1,a_i\}$ and $A_k^{f_k} = [n+1]
\setminus \{1,x\}\Rightarrow A_k^f \cup A_k^{f_i} \cup \{1\} = [n+1]
 \Rightarrow N(f, f_i, <1>) = \emptyset$. Here $\sigma = \{f, f_i\} \in
S'$ is of the type $(III)$.
\end{enumerate}

Finally, consider the case when $\tau \in A_2$. There exists $f = a_1\ldots a_n$, $a _{i} \neq 1\,\, \forall\,\, i \in [n]$,
such that $\tau$ = $\{<2>, <y>, f\}$, where $y \neq 1$.

For any $z \in [n+1] \setminus \{1,2, y\}$, $N(<z>)$ contains  $<1>, <2>$ and $<y>$ which implies
$\{<2>, <y>\} \in S_1$. Since $\sigma \notin S_1$, $\sigma $  is either  $\{f,<2>\}$ or $\{f,<y>\}$.

$A_i^{f}$ = $[n+1] \setminus \{1,2\}$, where $a_i = 2$ which implies $N(f, <2>, <1>) = \emptyset$
implying that $\{f, <2>\} \notin S_1$. Hence $\{f, <2>\} \in S'$ is of the form $(I)$.

 $A_j^f$ = $[n+1] \setminus \{1, y\}$ , where $a_j = y$ which implies $N(f,<y>, <1>) = \emptyset$ thereby
showing $\{f, <y>\} \in S'$ and $\sigma$ is of the type $(I)$.
\end{proof}

Let $S_2$ be the set of all  the 1-cells in $S'$ except those of the type $\{f, <2>\}$, $ a_i \neq 1,\,1\leq i\leq n$.
 We now define the map $\mu_2 : S_2 \longrightarrow P \setminus S_2$, as\\
(i) $\mu_2(\{f, f_i\}) = \begin{cases}
                    \{f, f_i, <x>\},  & \text{if $x > a_i$} \\
                    \{f, f_i, <a_i>\},  & \text{if $x < a_i$},
                   \end{cases}$\\
where  $a_k=1$, $\text{Im}\,f = [n+1]\setminus \{x\}$
and $i \neq k$.\\
(ii)  $\mu_2(\{f, <x>\}) =\{f,f_k, <x>\}$ where $a_k=1$, $\text{Im}\,f = [n+1]\setminus \{x\}$,  \\
(iii) $\mu_2(\{f, <y>\}) = \{f,<y>, <2>\}$ where  $y \neq 1, 2$ and
$1 \notin$ Im $f$.

\begin{claim}\label{acm}  $\mu_2$ is injective.

From the definition of $\mu_2$, for any $\sigma\in S_2$, dim($\mu_2(\sigma)$) = $2$ and
$\sigma \subset \mu_2(\sigma)$.
Therefore $\mu_2(\sigma) \succ \sigma$, for each $\sigma$.
Let $\mu_2(\sigma_1)$ = $\mu_2(\sigma_2)$ = $\tau$ for some $\sigma_1, \sigma_2 \in S_2$. There are three possibilities for $\tau \in \text{Im}\,\mu_2$.

 \begin{enumerate}
  \item $\tau = \{ f, f_i, <x>\}$,  $x \notin \text{Im}\,f$, $a_k= 1$, $i \neq k$, $x > a_i$.

$\{f_{i}, <x>\}\in S_1$ (since  $\{f_i, <x>, <1>\} \in N(<a_i>)$) and $\{f,  <x>\}\in S_2$ imply that both $\mu(\{f_{i}, <x>\})$ and
$\mu (\{f,  <x>\}$ are not equal to $\tau$. Hence $\sigma_1=\sigma_2=\{f, f_i\}$.

If $x<a_i$, then $\tau = \{ f, f_i, <a_i>\} = \{ f_i,(f_i)_i, <a_i>\}$ and the same argument as the one above holds.

 \item $\tau = \{f, f_k, <x>\}$, $x \notin \text{Im}\, f$, $a_k = 1$.

 $1 \notin \text{Im}\, f_k  \Rightarrow \{f, f_k, <1>\} \in N(f_k)  \Rightarrow \{f, f_k\} \in S_1  \Rightarrow \sigma_1, \sigma_2 \neq \{f, f_k\}$.

 If  $x \neq  2$, then $\mu_2(\{f_k, <x>\})$ = $\{f_k, <x>, <2>\} \neq \tau$ and when $x = 2$, then $\{f_k, <2>\} \notin S_2$.
 Hence both $\sigma_1$ and  $\sigma_2$  have to be $\{f, <x>\}$.

 \item $\tau$ = $\{f, <2>, <y>\}$,  $y \in \text{Im}\, f$, $y \neq 1$, $1\notin \text{Im}\,f$.

 Since $\{f, <2>\} \notin S_2$ and $\{<2>, <y>\} \in S_1$, $\sigma_1$ = $\sigma_2$ = $\{f, <y>\}$.
 \end{enumerate}
\end{claim}
From Claim \ref{acm},  $\mu_2 : S_2 \longrightarrow P \setminus S_2$ is a partial matching.

Since $S_2\cap \mu_{1}(S_1), \mu_1(S_1) \cap \mu_2(S_2) = \emptyset$, the map $\mu: S_1 \cup S_2
\longrightarrow P$ defined by $\mu_1$ on $S_1$ and $\mu_2$ on $S_2$ is a partial matching on $P$. This map is well defined since $S_1\cap S_2=\emptyset$.

\begin{lem}
 $\mu$ is an acyclic matching.
\end{lem}
 \begin{proof} Let $C$ = $P \setminus \{S_1, S_2, \mu_1(S_1), \mu_2(S_2)\}$. Suppose there exists a sequence of cells
$\sigma_1, \sigma_2, \ldots, \sigma_t$ $\in P \setminus C$ such that
$\mu(\sigma_1)\succ \sigma_2, \mu(\sigma_2)\succ \sigma_3, \ldots ,
\mu(\sigma_t) \succ \sigma_1$. If $\sigma_i \in S_1$, then
$\mu(\sigma_i) = \sigma_i \cup <1> \succ \sigma_{i+1(mod \hspace{0.1
cm} t)}$ which implies $<1> \in \sigma_{i+1(mod \hspace{0.1 cm}
t)}$, which is not possible by the construction of $S_1$ and $S_2$.
Hence $\sigma_i$ has to be in $S_2$, for each $i$, $1 \leq i \leq
t$. From Lemma \ref{4.7}, $\sigma$ has the following three forms.

\begin{enumerate}
 \item $\sigma_i = \{f,<y>\}$, $1 \notin \text{Im}\, f$, $y \neq 1$.

Since $\{f, <2>\} \notin S_1\cup S_2 \cup \mu(S_1) \cup \mu(S_2)$,
$y \neq 2$. Further,
 $\mu(\sigma_i) =\{f, <y>, <2>\}$ shows that $\sigma_{i+1}$ has to be $\{<y>, <2>\}$ or $\{f, <2>\}$, both of which are impossible.
Hence, $\sigma_i$ is not of this form.

\item  $\sigma_i = \{f,<x>\}$, $x \notin \text{Im}\, f$, $x \neq 1$, $a_k = 1$.

$\mu(\sigma_i)$ = $\{f, f_k, <x>\}$ implies that $\sigma_{i+1}$ has
to be either $\{f_k, <x>\}$ or $\{f, f_k\}$. Since $1 \notin$ Im
$f_k$ and $x\neq 1$, $\sigma_{i+1} \neq \{f_k, <x>\}$ (from the
above case). Further,  $\{f, f_k, <1> \}\in N(f_k)$ implies that
$\{f,f_k\}$ is an element of $S_1$ and therefore this case too is
not possible.

\item  $\sigma_i$ = $\{f,f_i\}$, $a_k=1$, $\text{Im}\, f=[n+1]\setminus\{x\}$, $i \neq k$.
\begin{center}
 $\mu(\sigma_i) = \mu_2(\sigma_i) = \begin{cases}
                    \{f, f_i, <x>\},  & \text{if $x > a_i$} \\
                    \{f, f_i, <a_i>\},  & \text{if $ x < a_i$}.
                   \end{cases}$
\end{center}
If $x > a_i$, case (2) shows that $\sigma_{i+1} \neq \{f,<x>\}$.
Since $a_i \neq 1, x$, and
$a_i \notin$ Im $f_i$, $\{f_i, <1>, <x>\} \in N(<a_i>)$ which implies that  $\{f_i, <x>\} \in S_1$.
A similar argument shows that the case $x < a_i$ is also not possible.
\end{enumerate}

Our assumption that the above sequence $\sigma_1, \sigma_2, \ldots, \sigma_t$ exists is wrong. Therefore $(S,\mu)$ where $S = S_1 \cup S_2$, is an acyclic matching.
\end{proof}

Every  element of $C$ is a critical cell corresponding to this matching. We now describe the structure of the elements of $C$.

For any 0-cell $<x> \ \neq \ <1>$ in $\Delta$, if $y \neq  1,
x$ then $\{<x>, <1>\} \in N(<y>)$, thereby implying that $<x>$ $\in
S_1$. If $f$ is a 0-cell with $|\text{Im}\, f|= n$ then either $f,
<1>$ $\in N(<[n+1] \setminus$ Im $f >)$ or $<1>, f \in N(f)$
accordingly as $1 \in \text{Im}\, f$ or not. In both cases $f \in
S_1$ and therefore $<1>$ is the only critical 0-cell.

Any 2-cell $\sigma$ of $\Delta$ belongs to $M_1 \cup A_1 \cup A_2
\cup A_3$. Each element of $A_1$ belongs to $\mu_{1}(S_1)$. If $\sigma \in A_2$, then
$\sigma = \{f,<2>, <y>\}$ with $\text{Im}\,f = [n+1] \setminus \{1\}$,
$ y \neq 1, 2$. Since $\mu_2(\{f, <y>\})=\sigma$, $\sigma \notin C$. Therefore, if
$\sigma$ has to be a critical 2-cell, then $\sigma$ has to belong to either $M_1$ or
$A_3$.

If $\sigma \in M_1$, then $\sigma = \{f, f_i, <x>\}$ and $ x \notin \text{Im}\, f$. Clearly, $x\neq 1$.
Let $a_k = 1$. $\mu (\{f, <x>\}) = \{f, f_k, <x>\}$, $i \neq k$.
If $x > a_i$, then $\mu(\{f, f_i\}) = \sigma$.
If $x< a_i$ then $\{f_i, <1>, <x> \}\in N(<a_i>)$ which implies that
$\sigma \notin \mu(S)$.  Further, $\sigma \notin S_1, S_2$ and therefore
 $\sigma \in M_1$
will be a critical 2-cell if and only
if $\sigma = \{ f, f_i, <x>\}$, where $x \notin \text{Im}\, f$, $a_k = 1$, $i \neq k$ and
$x < a_i$.

Finally from $A_3$, there exists exactly one critical cell $\{<y_1>, <y_2>, \ldots, <y_n>\}$ of dimension $n-1$,
where $y_1, y_2, \ldots, y_n
\in [n+1] \setminus \{1\}$ (since any proper subset of $\{<y_1>, <y_2>, \ldots, <y_n>\}$ belongs to $S_1$).

Therefore the set of critical cells $C$ = $\{<1>\}$ $\cup$ $C_1 \cup C_2 \cup C_3$, where
\begin{eqnarray*}
 &&C_1 = \{ \{f,<2>\}\ | \ \mbox{Im}\ f = [n+1] \setminus \{1\} \}, \\ 
 &&C_2 =  \{ \{f, f_i, <x>\}\ | \ \mbox{Im}\ f = [n+1] \setminus \{x\}, a_k = 1, i \neq k \ \mbox{and} \
 x < a_i\},\\
 &&C_3 = \{ \{<y_1>, <y_2>, \ldots, <y_n>\} \ | \ [n+1] \setminus \{1\} = \{ y_1, y_2, \ldots, y_n\} \}.
\end{eqnarray*}

Hence the critical cells are of dimension $0,1,2$ and $n-1$.

Clearly $| C_1 | = n!$ and $| C_3 |$ = $1$. For each fixed $x \neq 1$, let
$r= |\{s \in [n+1]\,|  s > x\} |$ and $Q = \{f \in V(G)\,|\,\text{Im}\,f \,=[n+1] \setminus \{x\} \}$. The cardinality of $Q$ is
$n!$. For $f \in Q$, $\{f, f_i, <x>\} \in C_2$ if and only if $x < f(i)$. Hence $|\{\tau = \{f, f_i, <x>\}$ $|$
$ \{x\} = [n+1] \setminus$ Im $f, \tau \in C_2\}|$ = $r(n!)$.
Therefore  $| C_2 |$ = $n! \sum\limits_{r=1}^{n-1} r = \frac{n! (n-1) n} {2}.$

We now describe the Morse Complex $\M= (\M_{i}, \partial)$ corresponding to this acyclic matching on the poset $P$.
If $c_i$ denotes the number of critical $i$ cells of $C$,
then the free abelian group generated by these critical cells is denoted by $\mathcal{M}_i$.
Our objective now is to first compute the $\Z_{2}$ homology groups of the Morse complex $\M$. We use the following version
of Theorem \ref{dmt}, from which we explicitly compute the boundary maps in the Morse Complex $\M$.

\begin{prop} \label{prop6}(Theorem 11.13 \cite{dk})

Let $X$ be a simplicial complex and $\mu$ be an acyclic matching
on the face poset of $X\setminus \{\emptyset\}$. Let $c_i$ denote
the number of critical $i$ cells of $X$. Then
\begin{itemize}
 \item [(a)] $X$ is homotopy equivalent to $X_c$, where $X_c$ is a $CW$ complex with $c_i$ cells in dimension $i$.
 \item [(b)] There is a natural indexing of cells of $X_c$ with the critical cells of $X$ such that for any two cells
 $\tau$ and $\sigma$ of $X_c$ satisfying dim $\tau$ = dim $\sigma + 1$, the incidence number $[\tau: \sigma]$
 is given by
 \begin{center}
  $[\tau:  \sigma ]$ = $\sum\limits_{c} w(c).$
\end{center}
The sum is taken over all (alternating) paths $c$ connecting $\tau$ with $\sigma$
i.e., over all sequences $c$ = $\{\tau, x_1, \mu(x_1 ), \ldots, x_t, \mu(x_t ), \sigma \}$ such that $\tau \succ x_1$,
$\mu(x_t) \succ \sigma$, and $\mu(x_i) \succ x_{i+1}$ for $i = 1, \ldots, t-1$. The quantity $w(c)$ associated to this
alternating
path is defined by
\begin{eqnarray*}
 w(c) := (-1)^t [\tau: \sigma] [\mu(x_t): \sigma] \prod\limits_{i=1}^{t}[ \mu(x_i): x_i]
 \prod\limits_{i=1}^{t-1} [\mu(x_i): x_{i+1}]
\end{eqnarray*}
where all the incidence numbers are taken in the complex $X$.
\end{itemize}
\end{prop}

We now determine all the possible alternating paths between any two critical cells.
\begin{lem} \label{lemma6}
 Let $\gamma \in \Delta$ be a $k$-simplex, $k>0$, such that $<1>\in \gamma$. Then $\gamma$ does not belong to
 any alternating path connecting two critical cells.
\end{lem}
\begin{proof} Given two critical cells $\tau$ and $\sigma$, let $c$ = $\{\tau, x_1, \mu(x_1), \ldots, x_t,
\mu(x_t)$, $\sigma$ $\}$ be an alternating path and let $\gamma \in c$.
 Since $<1> \ \in \gamma$, $\gamma \in \mu_1(S_1)$,
and therefore $\gamma \neq \tau, \sigma$. For some $i \in [t-1]$,
there exists $x_i \in c$ such that
 $\gamma $ = $\mu(x_i)$. Since $[\mu(x_i): x_i] = \pm 1$, $x_i$ has to be $\gamma \setminus <1>$. Any facet of $\gamma$
 different from $x_i$ must contain $<1>$. But since $x_{i+1}<\mu(x_i)$, $x_{i+1}$ has to be a facet of $\mu(x_i)$
 and therefore must belong to $S$, which is impossible, as $<1>$ $\in x_{i+1}$ implies $x_{i+1} \in \mu(S)$ and
 $\mu(S) \cap S = \emptyset$. Hence $\gamma \notin c$.
 \end{proof}

\begin{lem} \label{lemma7}
  Let  $\tau$ = $\{f, f_i, <x>\}$ be a critical 2-cell with $i \neq k$ and $a_k = 1$. There exists exactly one alternating
  path from  $\tau$ to each of exactly 2 critical 1-cells $\alpha = \{f_k, <2>\}$ and $\beta = \{(f_i)_k, <2>\}$.
\end{lem}

\begin{proof}

Let $\tau =\{f, f_{i}, <x>\} \in C_2$ be a critical 2-cell.
For any alternating path $c$ from $\tau$ to a critical $1$ cell,
$\tau \succ x_1$, {\i.e.} $x_1$ is a facet of $\tau.$ We have three choices for $x_1$.

\begin{enumerate}
\item  $x_1 = \{f, <x>\}$.

 Since  $x_2$ has to be a facet of $\mu(x_1)=\{f, f_k, <x>\}$, it is either $\{f, f_k\}$ or $\{f_k, <x>\}$.
  In the former case, $\mu(\{f, f_k\}) =\{f, f_k, <1>\}$ (since $\{f, f_k, <1>\} \in N(f_k))$ which contradicts Lemma \ref{lemma6}.

If $x=2$, then $x_2=\{f_k, <2>\}$ is  a critical 1- cell and the alternating path is  $\{ \tau, x_1 = \{f,<2>\},  \{f,f_k, <2>\}, \{f_k, <2>\}\}$.

If $x > 2$, then $\mu(x_2) = \{f_k, <x>, <2>\}$ and $x_3 $ has to be
the critical 1-cell $\{f_k, <2>\}$  (since $\{<x>, <2>, <1>\} \in
N(<y>)$, $y \neq 1,2,x$). The alternating
path is $\{\tau, \{f,<x>\}, \{f,<x>, f_k\}, \{f_k, <x>\}, \{f_k,<x>, <2>\}, \{f_k, <2>\}\}$.

\item $x_1$ = $\{f, f_i\}$.

$\mu (x_1) = \{f, f_i, <a_i>\}$ forces $x_2$ to be $\{f_i, <a_i>\}$ (as $\{f, <a_i>, <1>\} \in N(<x>) \Rightarrow x_2\neq\{f, <a_i>\} $).
 $\mu(x_2) = \{f_i, <a_i>, (f_i)_k\}$  implies that $x_3 = \{f_{i}, (f_i)_k\} \,\text{or}\, \{(f_i)_k, ,<a_i>\}$.
But, $\{f_i, (f_i)_k, <1>\} \in N((f_i)_k)$ shows that $x_3 = \{(f_i)_k, <a_i>\}$. Since $x < a_i$ and $x \neq 1$, $a_i$ is not
 2 and thus $\mu(x_3) = \{(f_i)_k,<a_i>, <2>\}$. Since  $\{<a_i>, <2>, <1>\}$ is a simplex in $\Delta$,  $x_4$ has to be the critical cell $\{(f_i)_k, <2>\}$.
The alternating path is
$\{\tau, \{f, f_i\}, \{f, f_i, <a_i>\}, \{f_i, <a_i>\}, \{f_i, <a_i>, (f_i)_k\},
\{(f_i)_k, <a_i>\}, \{(f_i)_k,<a_i>, <2>\}, \{(f_i)_k, <2>\}\}$.

\item  $x_1$ = $\{f_i, <x>\}$.\\
Since $a_i \notin \text{Im}\, f_i$ and $a_i \neq 1$, $\{f_i, <x>,
<1>\} \in N(<a_i>)$. Thus $x_1$ can not be an element of $c$.
\end{enumerate}
Hence, for each critical 2-cell $\tau$ = $\{f, f_i, <x>\} \in C_2$, there exist unique alternating paths from $\tau$
to exactly 2 critical 1-cells.
 \end{proof}

Consider $\tau = \{<y_1>, <y_2>, \ldots, <y_n>\} \in C_3$. There exists no alternating path from $\tau$ to any critical cell because each facet of $\tau$ belongs to $S_1$.

 If $\alpha = \{f, <2>\}$ is a critical 1-cell, then $1 \notin$ Im $f$. Since $n \geq 3$, there exists $i \neq k $ such that $a_k \neq 2$ and $a_i= n+1$. Since $a_k <a_i$, the 2-cell $\{f_k, (f_k)_i, <f(k)>\}$ is a critical cell. From Lemma \ref{lemma7}, there exists an alternating path between these two cells, showing that there exists at least one alternating path to each critical 1-cell.

Let $W_n=\{a_1 a_2 \ldots a_n\in V(G)\,|\{a_1,a_2, \ldots, a_n \} =
[n+1] \setminus \{1\} \}$, where $n \in \mathbb{N}$, $n \geq 2$.
 Define a relation  $\sim$ on $W_n$  by, $a \sim b \iff
    \exists\, i ,j \in [n], i \neq j$  such that $a_i = b_j$, $a_j = b_i$ and $a_k = b_k$ for all
    $k \neq i, j$. The cardinality of $W_n$ is easily seen to be $n!$.

    \begin{lem} \label{lemma8}
   The $n!$ elements of $W_n$, $ \alpha = \alpha_1, \alpha_2, \ldots,\alpha_{n!}$ can be ordered
    in such a way that $\alpha_i \sim \alpha_{i+1}$, for $1 \leq i \leq n!-1$.

    \end{lem}
    \begin{proof} If $n = 2$, then $W_n$ = $\{23, 32\}$ and $23\sim 32$.
     Let us assume that $n \geq 3$. The proof is by induction on $n$.

Let $W_{n,i} =\{f \in W_n \ | \ a_1=i\}$. Clearly $W_n = \bigcup
\limits_{i=2}^{n+1} W_{n,i}$, where each $W_{n,i}$ is in bijective
correspondence with $W_{n-1}$. By the inductive hypothesis,  assume
that $W_{n-1}$ has the required ordering
$\alpha_{1}\sim\alpha_{2}\sim\dots\sim\alpha_{(n-1)!}$, where
$\alpha_{1} = a_1a_2\dots a_{n-1}$, $a_{i} \in \{2,3, \dots n\}$.
For a fixed first element $iw_{2}w_{3}\dots w_{n}\in W_{n,i}$, the
map $\phi_{i} : [n+1]\setminus \{1,i\} \rightarrow \{2,3,\dots n\}$ defined by
$\phi_{i} (w_{j})=a_{j-1}$ is bijective. Using the ordering in
$W_{n-1}$ and the map $\phi_{i}$, we get an ordering in $W_{n,i}$.
Beginning with $23\dots n+1=2w_2\dots w_n \in W_{n,2}$ and using the
map $\phi_{2}$ we order $W_{n,2}$. Let $2w_{2}'w_{3}'\dots w_{n}'$
be the last element of this ordering and $w_j'=3$. Then,
$2w_{2}'w_{3}'\dots w_{n} \sim 3w_{2}'\dots w_{j-1}'2w_{j+1}'w_n'$.
Using the map $\phi_{3}$ in the above method, we get an ordering for
$W_{n,3}$.
 Repeating this argument for $4\leq i\leq n+1$, we have the required ordering in $W_n$.
 \end{proof}

    Since every critical 1-cell contains $<2>$, henceforth a critical 1-cell $\{f, <2>\}$ shall be denoted
    by $f$.

 \begin{rmk} \label{remark2}
   There exists alternating paths from a critical 2-cell $\tau$ to $\alpha$ and $\beta$ if and only if $\alpha \sim \beta$.
    \end{rmk}

The set of critical 1-cells $C_1 = \{\alpha_i = \{f_i, <2>\} | f_i \in W_n\}$ is in bijective correspondence with $W_n$.
From Lemma \ref{lemma8},  we have an ordering $\alpha_1 \sim \alpha_2 \sim \ldots \sim \alpha_{n!}$ of the elements of $C_1$.  Let $C_2$ = $\{\tau_1, \tau_2, \ldots, \tau_{ \frac{n! (n-1) n} {2}}\}$ and
$A = [a_{ij}]$ be a matrix of order  $|C_1| \times |C_2|$, where
$a_{ij} = 1$, if there exists an alternating path from $\tau_j$ to $\alpha_i$ and $0$
if no such path exists.  Using Lemma \ref{lemma7}  each column of $A$ contains exactly two
 non zero elements which are 1. The rows of the matrix $A$ are
denoted by $R_{\alpha_i}$ and the columns are denoted by $C_{\tau_i}$.

\begin{lem} \label{lemma9}
 The set $B = \{R_{\alpha_2}, \ldots , R_{\alpha_{n!}}\}$ is a basis for the row space of $A$
 over the field $\mathbb{Z}_2$.
  \end{lem}

  \begin{proof} In each column exactly two entries are 1 and all other entries are $0$
  and thus column  sum is zero (mod $2$) and hence rank($A$) $< n!$.

 Assume $\sum\limits_{i=2}^{n!} a_i R_{\alpha_i}= 0$, $a_i \in \{0,1\}$.
  For $1 \leq i \leq n!-1$, let $\tau_{i}$ be the critical 2-cell which has alternating paths
  to $\alpha_i$ and $\alpha_{i+1}$. The column
  $C_{\tau_i}$ has the $i$ and $(i+1)^{th}$ entry equal to 1 and all other
  entries equal to zero.  $\sum\limits_{i=2}^{n!} a_i R_{\alpha_i}= 0$, implies
  $a_2= a_2+a_3= a_3+ a_4= \ldots= a_{(n-1)!}+ a_{n!} = 0$.
  Hence $a_2 = a_3= \ldots a_{n!} = 0$ and $B$ is a basis for the row space of $A$.
  \end{proof}

  Let the Discrete Morse Complex corresponding to the acyclic matching $\mu$ on $\Delta$ be
  $\M = (\M_n , \partial_n)$, $n \geq 0$ where $\M_i$ denotes the free abelian groups over $\mathbb{Z}_2$ generated
  by the critical $i$-cells. The only non trivial groups are $\M_0, \M_1, \M_2, \M_{n-1}$.
For any two critical cells $\tau$ and $\sigma$ such that dim($\tau$) = dim($\sigma)$ +1, the incidence number
$[\tau: \sigma]$ is either 0 or 1.

We have developed all the necessary tools to prove the main results.

 \vspace{.2in}

\noindent {\bf Proof of Theorem \ref{thm1}.}

\begin{proof}
The graph $K_{n+1}^{K_n}$ folds to graph $G$, by Lemma \ref{lemma2}
and therefore Hom $(K_2, K_{n+1}^{K_n}) \simeq$ Hom $(K_2,G)$.
Further since Hom $(K_2, G) \simeq \mathcal{N}(G)$ and
$\mathcal{N}(G) \simeq \Delta$, from Proposition \ref{prop6}, it is
sufficient to compute the homology groups of the Morse Complex $\M$.

For all $y \neq z \in [n]$, $\{<y>, <z>\} \in \mathcal{N}(G)$, thereby showing that $\{<y>, <z>\}$ is an edge
in $G$. If $f \in V(G)$ such that $|\text{Im}\,f| = n$, then $\{f, <x>\}$ is an edge, where $x \notin \text{Im}\,f$.
Since $\{<x>, <y>\}$ is an edge for all $x, y \in [n+1]$, any two vertices of $G$ are connected by an edge path and therefore $\mathcal{N}(G)$ is connected  which implies
$H_0(\mathcal{N}(G); \mathbb{Z}_2) \cong \mathbb{Z}_2$.

Since $H_0(\mathcal{N}(G); \mathbb{Z}_2) \cong \mathbb{Z}_2
 \cong (\M_0, \mathbb{Z}_2)$, $\text{Ker}\, \partial_1 \cong \mathbb{Z}_2^{n!}$, where $\partial _1: \M_1 \rightarrow \M_0$ is a boundary map.

 Since $n \geq 4$, any critical 2-cell belongs to
 $C_2$. Further since any critical 2-cell is connected by alternating paths
 to exactly two 1-cells, from Lemma \ref{lemma9}, the rank of the group homomorphism
 $\partial_2 :\mathbb{Z}_2^p \longrightarrow \mathbb{Z}_2^{n!}$ is
 $n! -1$. Therefore $H_1(\M; \mathbb{Z}_2) \cong \mathbb{Z}_2$.

 If $n = 4$, then $\tau$ = $\{<2>, <3>, <4>, <5>\}$ is the only critical 3-cell and
 $\M_3 \cong \mathbb{Z}_2$. Since each facet of $\tau$ belongs to $S_1$, there will be no alternating path from $\tau$
 to any critical 2-cell which implies that the incidence number $[\tau: \alpha ] = 0$, for any critical
2-cell $\alpha$ and $\partial_3: \M_{3} \rightarrow \M_{2}$ is the zero map.  Rank($\partial_2$) = $n!-1$ and therefore $\text{Ker}\,\partial_2 \cong \mathbb{Z}_2^{p- n! +1}$. Thus $H_2(\M, \mathbb{Z}_2) \cong \mathbb{Z}_2^{p-n!+1}$.

If $n > 4$,  $\tau = \{<2>, <3>, \ldots, <n+1>\}$ is the only $n-1$ critical cell. $\M_{n-2}$ and $\M_{n}$ are trivial groups and therefore $H_{n-1}(\M, \mathbb{Z}_2) \cong \mathbb{Z}_2$.
\end{proof}
\begin{cor} \label{cor}\hspace{15 cm}
\begin{center}
            $H_k(\mathcal{N}(K_4^{K_3}); \mathbb{Z}_2) = \begin{cases}
            \mathbb{Z}_2 , & \text{if \hspace{0.3 cm}$k = 0, 1$}.\\
            \mathbb{Z}_2^{14}, & \text{if \hspace{0.3 cm} $k = 2$}\\
                             0, & \text{otherwise}.\\
                       \end{cases}$
  \end{center}

\end{cor}
\begin{proof} Since $n= 3$ in this case, $|C_1| = 6$, $|C_2|$ = $18$ and $C_3 = \{\{<2>, <3>, <4>\}\}$.
There exist $19$ critical
2-cells and therefore $\M_2$ $\cong \mathbb{Z}_2^{19}$. Since $\mathcal{N}(K_4^{K_3})$ is path connected,
$H_0(\mathcal{N}(K_4^{K_3}); \mathbb{Z}_2) \cong \mathbb{Z}_2$.

Each facet of $\tau = \{<2>, <3>, <4>\}$ belongs to $S_1$ and
therefore there exists no path from $\tau$ to any critical 1-cell
$\alpha$ and therefore the incidence number $[\tau : \alpha] = 0$
for any critical 1-cell $\alpha$. Hence $\partial_2(\tau) = 0$ i.e.
$\tau \in$ Ker $\partial_2$. From Lemma \ref{lemma9}, rank
($\partial_2$) = $5$. Since $H_0(\mathcal{N}(K_4^{K_3});
\mathbb{Z}_2) \cong \mathbb{Z}_2$ = $\M_0$, $\partial_1 = 0$.
Therefore $H_1(\mathcal{N}(K_4^{K_3}); \mathbb{Z}_2) \cong
\mathbb{Z}_2$.

The rank of $\partial_2 = 5$ shows that Ker $\partial_2$ $\cong$
$\mathbb{Z}_2^{14}$. Further there is no critical cell of dimension
greater than $2$, $\M_i = 0$, for all $i > 2$. Hence
$H_2(\mathcal{N}(K_4^{K_3}); \mathbb{Z}_2) \cong \mathbb{Z}_2^{14}$
and $H_k(\mathcal{N}(K_4^{K_3}); \mathbb{Z}_2) = 0$, for all $k >
2$.
\end{proof}

We recall the following result to prove Theorem \ref{thm2}.

\begin{prop} \label{prop7}(Theorem 3A.3, \cite{h}) \\
 If $C$ is a chain complex of free abelian groups, then there exist short exact sequences
 \begin{center}
  $0 \longrightarrow H_n(C; \mathbb{Z}) \otimes \mathbb{Z}_2 \longrightarrow H_n(C; \mathbb{Z}_2) \longrightarrow$
  Tor$(H_{n-1}(C; \mathbb{Z}), \mathbb{Z}_2) \longrightarrow 0$
 \end{center}
 for all n and these sequences split.

\end{prop}


\noindent {\bf  Proof of Theorem \ref{thm2}.}
\begin{proof}
 Since $\mathcal{N}(K_{n+1}^{K_n})$ is path connected, we only need to show that $\pi_1(\mathcal{N}$ $(K_{n+1}^{K_n}))
 \neq 0$.
 If $n = 2$, then $\mathcal{N}(K_{3}^{K_2}) \simeq$ Hom $(K_2 \times K_2, K_3) \simeq$ Hom $(K_2 \sqcup K_2, K_3)
 \simeq$ Hom $(K_2, K_3)$ $\times$ Hom $(K_2, K_3) \simeq S^1 \times S^1$.
 Hence $\pi_1(\mathcal{N}(K_3^{K_2}) \cong \mathbb{Z} \times \mathbb{Z}$.

 Now let us assume $n \geq 3$. Since $H_1(\mathcal{N}(K_{n+1}^{K_n}); \mathbb{Z})$ is abelianization of
 $\pi_1(\mathcal{N}(K_{n+1}^{K_n}))$,
 it is enough to show that $H_1(\mathcal{N}(K_{n+1}^{K_n}); \mathbb{Z})$ $\neq 0$.
 From Proposition \ref{prop7},
 $H_1(\mathcal{N}(K_{n+1}^{K_n}); \mathbb{Z}_2) \cong H_1(\mathcal{N}(K_{n+1}^{K_n});$ $\mathbb{Z}) \otimes
 \mathbb{Z}_2$ $\oplus$ Tor$(H_0(\mathcal{N}(K_{n+1}^{K_n});$ $\mathbb{Z}_2)$. Since
 $H_0(\mathcal{N}(K_{n+1}^{K_n}); \mathbb{Z}) \cong \mathbb{Z}$, Tor$(H_0(\mathcal{N}(K_{n+1}^{K_n}); \mathbb{Z}_2) = 0$.
 So $H_1(\mathcal{N}$ $(K_{n+1}^{K_n}); \mathbb{Z}) = 0$, implies that $H_1(\mathcal{N}(K_{n+1}^{K_n}); \mathbb{Z}_2) = 0$,
 which is a contradiction to Theorem \ref{thm1} and Corollary \ref{cor}.
 Therefore $H_1(\mathcal{N}(K_{n+1}^{K_n}); \mathbb{Z})$ $\neq 0$.
 \end{proof}

 The maximum degree $d$ of the graph $K_2 \times K_{n}$ is $n-1$ and
 Hom $(K_2 \times K_n, K_{n+1}) \simeq \mathcal{N}(K_{n+1}^{K_n})$. Hence Hom $(K_2 \times K_{n}, K_{n+1})$ is exactly
 $(n+1 - d -2)$-connected.

\vspace{.1in}

\noindent {\bf  Proof of Corollary \ref{cor1}.}
\begin{proof}
Theorem \ref{thm2} gives the result for the case $m = n+1$. If $m = n$, then for any $f \in V(K_n^{K_n})$
with $\text{Im}\,f = [n]$, $N(f) = \{f\}$. Since $n \geq 2$, $\mathcal{N}(K_n^{K_n})$ is disconnected.

If $m < n$, Lemma \ref{lemma2} shows that $K_m^{K_n}$  can be folded to the graph $G$, where
$V(G) = \{<x> \ | x \in [m]\}$. Then $N(<x>) = \{<y> \ | \ y \in [m] \setminus \{x\} \}$, for all $<x> \ \in V(G)$
and therefore $\mathcal{N}(G)$ is homotopic to the simplicial boundary of $(m-1)$-simplex.
Hence $\mathcal{N}(K_m^{K_n}) \simeq \mathcal{N}(G) \simeq S^{m-2}$.
Therefore $\text{conn(Hom}(K_m^{K_n})) = m-3$.

\end{proof}



\bibliographystyle{plain}

\end{document}